\documentclass[10pt]{article}
\usepackage{cite}
\usepackage{amsmath,amssymb,amsthm}
\usepackage{titlesec,hyperref}
\usepackage{color}
\usepackage{fancyhdr}
\usepackage[margin=3cm]{geometry}
\usepackage{ulem}
\newtheorem{theo}{Theorem}[section]
\newtheorem{lemm}[theo]{Lemma}
\newtheorem{defi}[theo]{Definition}

\newtheorem{prop}[theo]{Proposition}
\newtheorem{rema}[theo]{Remark}
\numberwithin{equation}{section}

\allowdisplaybreaks
\begin{document}
\title{The continuous dependence for the Navier-Stokes equations in $\dot{B}^{\frac{d}{p}-1}_{p,r}$}
\author{
Weikui $\mbox{Ye}^1$,\footnote{email: 904817751@qq.com}\quad
Wei $\mbox{Luo}^1$,\footnote{email: luowei23@mail2.sysu.edu.cn}
\quad and \quad
 Zhaoyang $\mbox{Yin}^{1,2}$\footnote{email: mcsyzy@mail.sysu.edu.cn}\quad\\
 $^1\mbox{Department}$ of Mathematics,
Sun Yat-sen University,\\ Guangzhou, 510275, China\\
$^2\mbox{Faculty}$ of Information Technology,\\ Macau University of Science and Technology, Macau, China}
\date{}
\maketitle
\begin{abstract}
 In this paper, we mainly investigate the Cauchy problem for the incompressible Navier-Stokes equations in homogeneous Besov spaces $\dot{B}^{\frac{d}{p}-1}_{p,r}$ with $1\leq p<\infty,\ 1\leq r\leq \infty, \ d\geq 2$. 
  Firstly, we prove the local existence of the solution and give a lower bound of the lifespan $T$ of the solution.  The lifespan depends on the Littlewood-Paley decomposition of the initial data, that is $\dot{\Delta}_j u_0$. Secondly, if the initial data $u^n_0\rightarrow u_0$ in $\dot{B}^{\frac{d}{p}-1}_{p,r}$, then the corresponding lifespan $T_n\rightarrow T$. Thirdly, we prove that the data-to-solutions map is continuous in $\dot{B}^{\frac{d}{p}-1}_{p,r}$. Therefore, the Cauchy problem of the Navier-Stokes equations is locally well-posed in the critical Besov spaces in the Hadamard sense. Moreover, we also obtain well-posedness and weak-strong uniqueness results in $L^{\infty}L^2\cap L^{2}\dot{H}^1$.
\end{abstract}


\tableofcontents
\section{Introduction}
In this paper, we mainly investigate the Cauchy problem of the incompressible Navier-Stokes ($NS$) equations:
\begin{equation}\label{nssp00}
\left\{\begin{array}{lll}
u_t-\Delta u+u\nabla u+\nabla P =0,\\
div~u=0\\
u|_{t=0}=u_0,
\end{array}\right.
\end{equation}
where the unknowns are the vector fields $u=(u_1,u_2,...,u_d)$ and the scalar function $P$. Here, $u$ is the velocity, respectively, while $P$ denotes the pressure.

In the seminal paper \cite{ns14}, J. Leray proved the global existence of finite energy weak solutions to the ($NS$) equations. Yet the uniqueness and regularity of such weak solutions are big open questions in
the field of mathematical fluid mechanics except the case when the initial data has special
structure. For instance, with axi-symmetric initial velocity and without swirl component,
O. A. Ladyzhenskaya \cite{ns13} and independently M. R. Ukhovskii and V. I. Yudovich \cite{ns21} proved the existence of
weak solution along with the uniqueness and regularity of such solution to the ($NS$) equations.
Another popular topic is the partial regularity of the weak solution and the typical work was established by Caffarelli, Kohn and Nirenberg\cite{Caffarelli}.
They proved that the singular set $S$ to the suitable weak solution satisfies that $P^1(S)=0$\footnote{$P^1$ is the one-dimensional Hausdorff measure defined by parabolic cylinders}. Lin\cite{Lin-F-H} proposed a simplified proof of the Caffarelli-Kohn-Nirenberg theorem.
The investigation on the weak solutions to \eqref{nssp00} is based on the following energy inequality:
$$\int u^2 dx+\int^T_0\int |\nabla u|^2 dx \leq \int u^2_0 dx, $$
for any $T>0$ and $u_0\in L^2$. The energy inequality implies the existence and  partial regularity of the weak solutions. However,
T. Buckmaster and V. Vical\cite{Buckmaster} proved that the weak solutions with bounded kinetic energy are not unique. This result means that the regularity of solutions is important to study the uniqueness.

In order to study the regularity, it is always transform the system \eqref{nssp00} into the following equations
\begin{equation}\label{nssp0}
\left\{\begin{array}{lll}
u_t-\Delta u =-\mathbb{P}(u\nabla u), \\
u|_{t=0}=u_0,
\end{array}\right.
\end{equation}
where $\mathbb{P}=I+\nabla (-\Delta)^{-1}div$ is the Leray project operator and the initial data satisfies $div~u_0=0$. Using the heat kernel and Duhamel's principle, one can rewrite the $NS$ equations into
$$u=e^{t\Delta}u_0-\int^t_0e^{(t-s)\Delta }\mathbb{P}div~(u\otimes u)ds.$$
The function $u$ satisfying the above equations is called a mild solution. The investigation for the mild solution is based on Banach's contraction principle. The key point is to look for the estimates for the bilinear operator
$$B(u,v)=\int^t_0e^{(t-s)\Delta }\mathbb{P}div~(u\otimes v)ds,$$
in some suitable Banach spaces $X$. If the norm $\|e^{t\Delta } u_0\|_X$ is small and $\|B(u,v)\|_X\leq C\|u\|_X\|v\|_X$, then one can prove the existence and uniqueness by virtue of the Banach contraction principle.

 Let's review some results about mild solutions with $d=3$. In the homogeneous Sobolev spaces, H. Fujita and T. Kato \cite{ns9} proved the global well-posedness of ($NS$) with small data $u_0\in \dot{H}^{\frac{1}{2}}$. In the Lebesgue space, T. Kato \cite{Kato} proved the global well-posedness of ($NS$) with small data $u_0\in L^3$.   This result was generalized by
M. Cannone, Y. Meyer and F. Planchon \cite{ns3} with small initial data $u_0\in \dot{B}^{\frac{3}{p}-1}_{p,\infty}$ with $p \in (3,\infty)$. The end-point result in this direction is due to H. Koch and D. Tataru \cite{ns11}, where they proved the global well-posedness of the ($NS$) equations with small data $u_0\in BMO^{-1}$.

On the other hand, J. Bourgain and N. Pavlovi\'{c} \cite{ns2} proved that the ($NS$) equation is ill-posed with initial data in $\dot{B}^{-1}_{\infty,\infty}$.
P. Germain in \cite{nsgerman} proved an instability result for $u_0\in \dot{B}^{-1}_{\infty,r}$ with $r > 2$ by showing that the data-to-solutions map is not in the class $C^2$. Recently, B. Wang \cite{wbx} proved that the ill-posedness for $u_0\in\dot{B}^{-1}_{\infty,r}$ with $1\leq r \leq 2$. Recently, the ill-posedness for the ($NS$) equation in $\dot{B}^{-1}_{\infty,r}$ ($1\leq r\leq\infty$) was studied. In \cite{illweizhi}, A. Cheskidov and R. Shvydkoy constructed a special initial data belonging to $B^{\frac{d}{p}-1}_{p,\infty}$ such that the solution is discontinuous at $t=0$ in the norm $\dot{B}^{-1}_{\infty,\infty}$. Since $B^{\frac{d}{p}-1}_{p,\infty}(p<d)\hookrightarrow\dot{B}^{\frac{d}{p}-1}_{p,\infty}\hookrightarrow\dot{B}^{-1}_{\infty,\infty}\hookrightarrow B^{-1}_{\infty,\infty}$, A. Cheskidov and R. Shvydkoy's result implied that \eqref{nssp00} is ill-posed in $\dot{B}^{\frac{d}{p}-1}_{p,\infty}$.

The majority of papers to study well-posedness or ill-posedness for the \eqref{nssp00} equations is based on small conditions.
For large initial data,  J. Y. Chemin, I. Gallagher and M. Paicu\cite{Chemin-Annals}
proved that the global well-posedness
requires that the initial data $u_0$ has a slow space variable (See also J. Y. Chemin and I. Gallagher\cite{ns5}, J. Y. Chemin and P. Zhang\cite{ns7}.).
For arbitrary initial data, Z. Guo, J. Li and Z. Yin in \cite{yin} proved the uniform continuous dependence and the inviscid limit of the \eqref{nssp00} equations with $u_0\in {B}^s_{p,r}$, $s>\frac{d}{p}+1$ or $s=\frac{d}{p}+1,r=1$.

In \cite{wbx},  B. Wang proposed that the largest homogeneous Besov space on the general initial data for which \eqref{nssp00} is well-posed (existence, uniqueness and continuous dependence) is still unknown, especially for the continuous dependence. In this paper, our aim
is to show that the continuous dependence for the mild solutions in the critical Besov spaces with arbitrary initial data.
Note that
$$\dot{B}^{\frac{d}{p}-1}_{p,r}(1\leq p,r<\infty)\hookrightarrow \dot{B}^{\frac{d}{p}-1}_{p,\infty}(1\leq p<\infty)\hookrightarrow BMO^{-1}\hookrightarrow\dot{B}^{-1}_{\infty,\infty}.$$
We are going to establish the local well-posedness of
 the Cauchy problem (\ref{nssp00}) in $\dot{B}^{\frac{d}{p}-1}_{p,r}$ ($1\leq p,r<\infty$). For $r=\infty$, we will assume that the initial data $u_0$ belongs to a subspace $\bar{B}^{\frac{d}{p}-1}_{p,\infty}$ of $\dot{B}^{\frac{d}{p}-1}_{p,\infty}$ to overcome the problem. This improves  considerably the previous results in \cite{chemin1999}.

For Leray's weak solutions, J. Y. Chemin \cite{chemin2011} proved the continuous dependence for $u_0\in L^2\cap\dot{H}^a\cap\bar{B}^{\frac{d}{p}-1}_{p,\infty}$ with $3<p<\infty$. Recently, T. Barker \cite{B1,B2} proved the continuous dependence for $u_0\in L^2\cap VMO^{-1}\cap\bar{B}^{s}_{p,p}$ with $3<p<\infty$ and $-1+\frac{2}{p}<s<0$. In this paper, we prove the continuous dependence for $u_0\in L^2\cap\dot{B}^{\frac{d}{p}-1}_{p,r}$ with $1\leq p<\infty$ and $1\leq r\leq\infty$ (when $r=\infty$, we also choose $\bar{B}^{\frac{d}{p}-1}_{p,\infty}$), which is different from previous results for more general initial data. In this way, if we give a little regularity on Leray's weak solutions, then the \eqref{nssp00} equations is well-posed.

The main difficulty is that the initial data maybe large and has low regularity. If the initial data $u_0\in B^s_{p,r}$ with $s$ large enough, then one can easily prove that the lifespan $T\approx \frac{1}{\|u_0\|^2_{B^s_{p,r}}}$. However, we consider the problem with $u_0\in \dot{B}^{\frac{d}{p}-1}_{p,r}$, and the lifespan $T$ depends not only on the Besov norm, but also on the Littlewood-Paley decomposition of $u_0$. In this paper, we would like to present a general functional framework to deal with the local existence of the solutions of \eqref{nssp00} in the homogeneous Besov spaces. In fact, we mainly prove that the lifespan $T$ of the \eqref{nssp00} equations can be solved by the following inequality
$$\|e^{t\Delta }u_0\|_{\widetilde{L}^{2}_T(\dot{B}^{\frac{d}{p}}_{p,r})\cap \widetilde{L}^{1}_T(\dot{B}^{\frac{d}{p}+1}_{p,r})}\leq a,$$
where $a$ is a small constant. In order to solve the above inequality, we have to
 use the Littlewood-Paley decomposition of the initial data $u_0=\sum_{j}\dot{\Delta}_j u_0$ and obtain the relationship between $T$ and the index $j$. Then we prove that if the initial data $u^n_0\rightarrow u_0$ in $\dot{B}^{\frac{d}{p}-1}_{p,r}$, then the corresponding lifespan satisfies $T_n\rightarrow T$, which implies that the common lower bound of the lifespan (see the key Lemma \ref{gj} below). Finally, we prove the continuous dependence in homogeneous Besov spaces.

Our paper is organized as follows. In Section 2, we give some useful preliminaries. In Section 3, we prove the local existence and the uniqueness of solutions to (\ref{nssp0}) with the expression of local time being given. In Section 4, we firstly prove that if the initial data $u^n_0$ tend to $u_0$ in $\dot{B}^{\frac{d}{p}-1}_{p,1}$, then their local existence times satisfy $T_n\rightarrow T$, which implies that they have public lower bound of the lifespan $T-\delta$ with $n$ being sufficient large. Then we obtain the continuous dependence in the critical Besov space. Finally, we prove the well-posedness of (\ref{nssp0}) corresponding to Leray's weak solutions if the initial data belongs to $L^2\cap\dot{B}^{\frac{d}{p}-1}_{p,r}$.

Our main result can be stated as follows:\\
\quad\\
\textbf{Well-posedness in Besov spaces:}\\
\begin{theo}\label{theorem}
Let $u_0\in \dot{B}^{\frac{d}{p}-1}_{p,1}(\mathbb{R}^d)$ with $r,p\in [1,\infty)$. Then there exists a positive time $T$ such that the Cauchy problem (\ref{nssp0}) is locally well-posed in $E^p_T$ with
$$E^p_T:= \widetilde{L}^{\infty}([0,T];\dot{B}^{\frac{d}{p}-1}_{p,r}(\mathbb{R}^d))\cap \widetilde{L}^1([0,T];\dot{B}^{\frac{d}{p}+1}_{p,r}(\mathbb{R}^d))$$
in the Hadamard sense. Moreover, we have $u\in C([0,T];\dot{B}^{\frac{d}{p}-1}_{p,r}(\mathbb{R}^d))$. If the initial data is small enough, then the Cauchy (\ref{nssp0}) is globally well-posed.
\end{theo}
\begin{defi}
We denote a closed subspace $\bar{B}^{\frac{d}{p}-1}_{p,\infty}(\mathbb{R}^d)$ of $\dot{B}^{\frac{d}{p}-1}_{p,\infty}$ as follows:
$$\bar{B}^{\frac{d}{p}-1}_{p,\infty}(\mathbb{R}^d):=\{f\in \dot{B}^{\frac{d}{p}-1}_{p,\infty}(\mathbb{R}^d)|\quad\lim_{|j|\rightarrow \infty}2^{\frac{d}{p}-1}\|\dot{\Delta}_jf\|_{L^p}=0\}.$$
\end{defi}
\begin{theo}\label{theorem2}
Let $u_0\in\bar{B}^{\frac{d}{p}-1}_{p,\infty}(\mathbb{R}^d)$ with $p\in [1,\infty)$. Then there exists a positive time $T$ such that the Cauchy problem (\ref{nssp0}) is locally well-posed in $$E_T:=\widetilde{L}^{\infty}([0,T];\dot{B}^{\frac{d}{p}-1}_{p,r}(\mathbb{R}^d))\cap \widetilde{L}^1([0,T];\dot{B}^{\frac{d}{p}+1}_{p,r}(\mathbb{R}^d))$$ in the Hadamard sense.
\end{theo}
\begin{rema}
In fact, one can find the existence and uniqueness results as for Theorem \ref{theorem} and Theorem \ref{theorem2} in \cite{chemin1999}, our main contribution is to prove the continuous dependence such that
\begin{align}\label{cd}
\|u-u^n\|_{L^{\infty}(\dot{B}^{\frac{d}{p}-1}_{p,r})}+\|u-u^n\|_{\widetilde{L}^{1}(\dot{B}^{\frac{d}{p}+1}_{p,r})}\leq C_{u_0}\|u_0-u^n_0\|_{\dot{B}^{\frac{d}{p}-1}_{p,r}},\quad 1\leq p<\infty,\ 1\leq r\leq \infty,
\end{align}
where the lifespan $T$ is independent of $n$.
\end{rema}
\quad\\
\textbf{The continuous dependence for the Leray weak solutions:}\\

The following theorem shows that the Leray weak solutions are continuously dependent on the initial data in a weak topology.
\begin{theo}\label{theorem3}
Let $div~u_0=div~u^n_0=0$, $u_0,u_0^n\in L^2(\mathbb{R}^d)\cap \dot{B}^{\frac{d}{p}-1}_{p,r}(\mathbb{R}^d)$ ($r=\infty$, $u_0\in L^2(\mathbb{R}^d)\cap \bar{B}^{\frac{d}{p}-1}_{p,\infty}(\mathbb{R}^d)$) with $p\in [1,\infty)$. If $u_0^n\rightarrow u_0 $ in $L^2\cap\dot{B}^{\frac{d}{p}-1}_{p,r}(\mathbb{R}^d)$, then there exists a positive time $T$ independent of $n$ such that
\begin{align}\label{L2}
\|u-u^n\|_{L^{\infty}(L^2)}+\|u-u^n\|_{L^{2}(\dot{H}^1)}\leq C\|u_0-u^n_0\|_{L^2},
\end{align}
where $u,u^n$ are the Leray weak solutions corresponding to the initial data $u_0$ and $u_0^n$. This implies the well-posedness of (\ref{nssp0}) for the Leray weak solutions in the Hadamard sense.
\end{theo}
\begin{rema}
Since the constant $C$ in \eqref{L2} is independent of the Besov norm of $u$, so the condition $u_0^n\rightarrow u_0 $ in $L^2\cap\dot{B}^{\frac{d}{p}-1}_{p,r}(\mathbb{R}^d)$ can be reduced to
$u_0^n\rightarrow u_0 $ in $L^2(\mathbb{R}^d)$ and there exists another $v_0$ such that $u_0^n\rightarrow v_0 $ in $\dot{B}^{\frac{d}{p}-1}_{p,r}(\mathbb{R}^d)$.
\end{rema}

\textbf{Notations: } Throughout the paper, we denote $\dot{B}^{s}_{p,r}(\mathbb{R}^d)=\dot{B}^{s}_{p,r}$, $\|u\|_{\dot{B}^{s}_{p,r}(\mathbb{R}^d)}+\|v\|_{\dot{B}^{s}_{p,r}(\mathbb{R}^d)}=\|u,v\|_{\dot{B}^{s}_{p,r}}$ and $C([0,T];\dot{B}^{s}_{p,r}(\mathbb{R}^d))=C_T(\dot{B}^{s}_{p,r})$, $L^p([0,T];\dot{B}^{s}_{p,r}(\mathbb{R}^d))=L^p_T(\dot{B}^{s}_{p,r})$.

\section{Preliminaries}
\par
In this section, we will recall some propositions and lemmas on the Littlewood-Paley decomposition and Besov spaces.

\begin{prop}\cite{book}
Let $\mathcal{C}$ be the annulus $\{\xi\in\mathbb{R}^d:\frac 3 4\leq|\xi|\leq\frac 8 3\}$. There exist radial functions $\chi$ and $\varphi$, valued in the interval $[0,1]$, belonging respectively to $\mathcal{D}(B(0,\frac 4 3))$ and $\mathcal{D}(\mathcal{C})$, and such that
$$ \forall\xi\in\mathbb{R}^d,\ \chi(\xi)+\sum_{j\geq 0}\varphi(2^{-j}\xi)=1, $$
$$ \forall\xi\in\mathbb{R}^d\backslash\{0\},\ \sum_{j\in\mathbb{Z}}\varphi(2^{-j}\xi)=1, $$
$$ |j-j'|\geq 2\Rightarrow\mathrm{Supp}\ \varphi(2^{-j}\cdot)\cap \mathrm{Supp}\ \varphi(2^{-j'}\cdot)=\emptyset, $$
$$ j\geq 1\Rightarrow\mathrm{Supp}\ \chi(\cdot)\cap \mathrm{Supp}\ \varphi(2^{-j}\cdot)=\emptyset. $$
The set $\widetilde{\mathcal{C}}=B(0,\frac 2 3)+\mathcal{C}$ is an annulus, and we have
$$ |j-j'|\geq 5\Rightarrow 2^{j}\mathcal{C}\cap 2^{j'}\widetilde{\mathcal{C}}=\emptyset. $$
Further, we have
$$ \forall\xi\in\mathbb{R}^d,\ \frac 1 2\leq\chi^2(\xi)+\sum_{j\geq 0}\varphi^2(2^{-j}\xi)\leq 1, $$
$$ \forall\xi\in\mathbb{R}^d\backslash\{0\},\ \frac 1 2\leq\sum_{j\in\mathbb{Z}}\varphi^2(2^{-j}\xi)\leq 1. $$
\end{prop}

\begin{defi}\cite{book}
Denote $\mathcal{F}$ by the Fourier transform and $\mathcal{F}^{-1}$ by its inverse.
Let $u$ be a tempered distribution in $\mathcal{S}'(\mathbb{R}^d)$. For all $j\in\mathbb{Z}$, define
$$
\Delta_j u=0\,\ \text{if}\,\ j\leq -2,\quad
\Delta_{-1} u=\mathcal{F}^{-1}(\chi\mathcal{F}u),\quad
\Delta_j u=\mathcal{F}^{-1}(\varphi(2^{-j}\cdot)\mathcal{F}u)\,\ \text{if}\,\ j\geq 0,\quad
S_j u=\sum_{j'<j}\Delta_{j'}u.
$$
Then the Littlewood-Paley decomposition is given as follows:
$$ u=\sum_{j\in\mathbb{Z}}\Delta_j u \quad \text{in}\ \mathcal{S}'(\mathbb{R}^d). $$

Let $s\in\mathbb{R},\ 1\leq p,r\leq\infty.$ The nonhomogeneous Besov space $B^s_{p,r}(\mathbb{R}^d)$ is defined by
$$ B^s_{p,r}=B^s_{p,r}(\mathbb{R}^d)=\{u\in S'(\mathbb{R}^d):\|u\|_{B^s_{p,r}(\mathbb{R}^d)}=\Big\|(2^{js}\|\Delta_j u\|_{L^p})_j \Big\|_{l^r(\mathbb{Z})}<\infty\}.$$

Similarly, we can define the homogeneous Besov space
$$\dot{B}^{s}_{p,r}=\dot{B}^{s}_{p,r}(\mathbb{R}^d):=\{u\in S'_h(\mathbb{R}^d) | \|u\|_{\dot{B}^{s}_{p,r}}:=\|2^{sj}\|\dot{\Delta}_ju\|_{L^p(\mathbb{S}^d)}\|_{l^r}\leq\infty\},$$
where the Littlewood-Paley operator $\dot{\Delta}_j$ is defined by
$$\dot{\Delta}_j u=\mathcal{F}^{-1}(\varphi(2^{-j}\cdot)\mathcal{F}u)\,\ \text{if}\,\ j\in\mathbb{Z}.$$
\end{defi}


\begin{defi}\cite{book}
Let $s\in\mathbb{R},1\leq p,q,r\leq\infty$ and $T\in (0,\infty].$ The functional space $\widetilde{L}^q_T(\dot{B}^{s}_{p,r})$ is defined as the set of all the distributions $f(t)$ satisfying
$\|f\|_{\widetilde{L}^q_T(\dot{B}^{s}_{p,r})}:=\|(2^{ks}\|\dot{\Delta}_kf(t)\|_{L^q_TL^p})_k\|_{l^r}<\infty .$
\end{defi}
By Minkowski's inequality, it is easy to find that
$$\|f\|_{\widetilde{L}^q_T(\dot{B}^{s}_{p,r})}\leq \|f\|_{L^q_T(\dot{B}^{s}_{p,r})}\quad q\leq r;\quad\quad\quad \|f\|_{\widetilde{L}^q_T(\dot{B}^{s}_{p,r})}\geq \|f\|_{L^q_T(\dot{B}^{s}_{p,r})}\quad q\geq r.$$
\quad\\

Finally, we state some useful results on the heat equation
\begin{equation}\label{s1cuchong}
\left\{\begin{array}{l}
    u_t+\Delta u=G,\ x\in\mathbb{R}^d,\ t>0, \\
    u(0,x)=u_0(x).
\end{array}\right.
\end{equation}

\begin{lemm}\label{heat}\cite{book}
Let $s\in\mathbb{R}, 1\leq q,q_1,p,r\leq\infty$ with $q_1\leq q$. Assume $u_0$ in $\dot{B}^s_{p,r}$, and $G$ in $\widetilde{L}^{q_1}_T(\dot{B}^s_{p,r})$. Then (\ref{heat}) has a unique solution $u$ in $\widetilde{L}^{q}_T(\dot{B}^{s+\frac{2}{q}}_{p,r})$ satisfying
$$ \|u\|_{\widetilde{L}^{q}_T(\dot{B}^{s+\frac{2}{q}}_{p,r})}\leq C\Big(\|u_0\|_{\dot{B}^s_{p,r}}+\|G\|_{\widetilde{L}^{q_1}_T\dot{B}^{s+\frac{2}{q_1}-2}_{p,r}}\Big). $$
In particular, if $s=\frac{d}{p}-1$, we have
$$ \|u\|_{\widetilde{L}^{\infty}_T(\dot{B}^{\frac{d}{p}-1}_{p,r})\cap \widetilde{L}^{\frac{4}{3}}_T(\dot{B}^{\frac{d}{p}-\frac{1}{2}}_{p,r})\cap \widetilde{L}^{2}_T(\dot{B}^{\frac{d}{p}}_{p,r})\cap \widetilde{L}^{4}_T(\dot{B}^{\frac{d}{p}+\frac{1}{2}}_{p,r})\cap \widetilde{L}^{1}_T(\dot{B}^{\frac{d}{p}+1}_{p,r})}\leq C\Big(\|u_0\|_{\dot{B}^{\frac{d}{p}-1}_{p,r}}+\|G\|_{\widetilde{L}^{1}_T\dot{B}^{\frac{d}{p}-1}_{p,r}}\Big). $$
\end{lemm}

\section{Local well-posedness}

We divide the proof of Theorem \ref{theorem} into 5 steps:
\quad\\
\textbf{Step 1: Constructing approximate solutions.}

Being different to the proof in \cite{ztjwdx1}, we will prove the local existence and uniqueness for (\ref{nssp0}) in a more precise way, which is necessary to prove the continuous dependence later. Now set $u_0\in \dot{B}^{\frac{d}{p}-1}_{p,r}(1\leq p<\infty)$ and define the first term $u^0:=e^{t\Delta}u_0$. Then we introduce a sequence $\{u^{n+1}\}$ with the initial data $u^n_0$ by solving the following heat equation:
\begin{equation}\label{sp1}
\left\{\begin{array}{lll}
u^{n+1}_t-\Delta u^{n+1}=\mathbb{P}(-u^{n}\nabla u^{n}),\\
divu^{n+1}=0,\\
u|_{t=0}=\dot{S}_nu_0,
\end{array}\right.
\end{equation}
where $\dot{S}_ng:=\sum_{k<n}\dot{\Delta}_k g$, it makes sense if $s<\frac{d}{p}$ or $s=\frac{d}{p},r=1$.
\quad\\
\textbf{Step 2: Uniform bounds.}
Taking advantage of Lemma \ref{heat}, we shall bound the approximating sequences in $E^p_T$. Now we claim that there exists some $T$ independent of $n$ such that the solutions $u^n$ satisfies the following inequalities:
\begin{align}\label{lsp2}
&(H_1):\quad \|u^{n}\|_{\widetilde{L}^{\infty}_T(B^{\frac{d}{p}-1}_{p,1})}\leq 2E_0,\notag\\
&(H_2):\quad \|u^{n}\|_{A}\leq 2a,\quad A:={\widetilde{L}^{2}_T(\dot{B}^{\frac{d}{p}}_{p,r})\cap \widetilde{L}^{1}_T(\dot{B}^{\frac{d}{p}+1}_{p,r})},
\end{align}
where $E_0:=\|u_0\|_{\dot{B}^{\frac{d}{p}-1}_{p,r}}$. \\
Now we suppose that $T$ and $a$ satisfy the following inequality:
\begin{align}\label{lsp2}
a\leq \min\{\frac{1}{16C^2_1E_0},(\frac{\sqrt{E_0}}{4C_1})^{\frac{2}{3}},\frac{1}{4C_1}\},
\end{align}
\begin{align}\label{lsp2.5}
\|e^{t\Delta}u_0\|_A\leq a,
\end{align}
where $C_1\geq C$, $C$ is the constant of Lemma \ref{heat}, we can choose $C_1$ much bigger if necessary.\\

It's easy to check that $(H_1)-(H_2)$ hold true for $n=0$. Now we show that if $(H_1)-(H_2)$ hold true for $n$, then they will hold true for $n+1$. In fact, combining (\ref{lsp2}) and Lemma \ref{heat}, we have
\begin{align}\label{lsp3}
&\|u^{n+1}\|_A \notag \\
&\leq \|e^{t\Delta}u_0\|_A+\|\mathbb{P}div(-u^{n}\otimes u^{n})\|_{\widetilde{L}^{1}_T(\dot{B}^{\frac{d}{p}-1}_{p,r})}\notag \\
&\leq a+C_1\|u^{n}\|_{\widetilde{L}^{4}_T(\dot{B}^{\frac{d}{p}-\frac{1}{2}}_{p,r})}\|u^{n}\|_{\widetilde{L}^{\frac{4}{3}}_T(\dot{B}^{\frac{d}{p}+\frac{1}{2}}_{p,r})}\notag \\
&\leq a+C_1\|u^{n}\|^{\frac{1}{2}}_{\widetilde{L}^{\infty}_T(\dot{B}^{\frac{d}{p}-1}_{p,r})}\|u^{n}\|^{\frac{1}{2}}_{\widetilde{L}^{2}_T(\dot{B}^{\frac{d}{p}}_{p,r})}
\|u^{n}\|^{\frac{1}{2}}_{\widetilde{L}^{2}_T(\dot{B}^{\frac{d}{p}}_{p,r})}\|u^{n}\|^{\frac{1}{2}}_{\widetilde{L}^{1}_T(\dot{B}^{\frac{d}{p}+1}_{p,r})}\notag \\
&\leq a+C_1(2E_0)^{\frac{1}{2}}(2a)^{\frac{3}{2}}\notag \\
&\leq 2a
\end{align}
and
\begin{align}\label{lsp4}
&\|u^{n+1}\|_{\widetilde{L}^{\infty}_T(\dot{B}^{\frac{d}{p}-1}_{p,r})} \notag \\
&\leq \|e^{t\Delta}u_0\|_{\widetilde{L}^{\infty}_T(\dot{B}^{\frac{d}{p}-1}_{p,r})}+\|\mathbb{P}div(-u^{n}\otimes u^{n})\|_{\widetilde{L}^{1}_T(\dot{B}^{\frac{d}{p}-1}_{p,r})}\notag \\
&\leq E_0+C_1(2E_0)^{\frac{1}{2}}(2a)^{\frac{3}{2}}\notag \\
&\leq 2E_0.
\end{align}
This implies $(H_1)-(H_2)$ hold true for $n+1$.

At the end of this step, we want to find the time $T$ and $a$ satisfying (\ref{lsp2}):\\
We discuss $\|u_0\|_{\dot{B}^{\frac{d}{p}-1}_{p,\infty}}$ in categories to meet the conditions $\uppercase\expandafter{\romannumeral1}$ and $\uppercase\expandafter{\romannumeral2}$ of (\ref{lsp2}):

(1) For $E_0\leq \frac{1}{16C_1}$, we let $a=2E_0$, which implies \eqref{lsp2}.\\
Then we have
$$\|e^{t\Delta}u_0\|_A\leq 2\|u_0\|_{\dot{B}^{\frac{d}{p}-1}_{p,r}}=2E_0=a.$$
That is (\ref{lsp2.5}) and we thus have
$$T=\infty.$$

(2) For $E_0\geq\frac{1}{16C_1}$, we let $a=\bar{c}:=\min\{\frac{1}{16C^2_1E_0},\frac{1}{4^{\frac{4}{3}}C_1}\}$, which implies \eqref{lsp2}. \\
Then let $T$ be small enough so that
$$\|e^{t\Delta}u_0\|_A\leq a.$$
In fact, since $u_0\in \dot{B}^{\frac{d}{p}-1}_{p,r}$, we let $j_0$ be an integer such that
\begin{align}\label{lsp5.5}
[\sum_{|j|\geq j_0}(\|\dot{\Delta}_ju_0\|_{L^p}2^{(\frac{d}{p}-1)j})^{r}]^{\frac{1}{r}}< \frac{a}{4}.
\end{align}

Then we have
\begin{align}\label{lsp6}
&\|e^{t\Delta}u_0\|_{\widetilde{L}^{1}_{T_0}\dot{B}^{\frac{d}{p}+1}_{p,r}}                                    \notag \\
&\leq [\sum_{j\in\mathbb{Z}}(2^{(\frac{d}{p}+1)j}\|\dot{\Delta}_ju_0\|_{L^p}\int_{0}^{T_0}e^{-2^{2^j}t}dt)^{r}]^{\frac{1}{r}}\notag \\
&\leq [\sum_{|j|\leq j_0}(2^{(\frac{d}{p}+1)j}\|\dot{\Delta}_ju_0\|_{L^p}T_0)^{r}]^{\frac{1}{r}}
+[\sum_{|j|> j_0}(2^{(\frac{d}{p}-1)j}(1-e^{2^{2j}T_0})\|\dot{\Delta}_ju_0\|_{L^p})^{r}]^{\frac{1}{r}}\notag \\
&\leq 2^{2j_0}{T_0}\|u_0\|_{\dot{B}^{\frac{d}{p}-1}_{p,r}}+ [\sum_{|j|> j_0}(2^{(\frac{d}{p}-1)j}\|\dot{\Delta}_ju_0\|_{L^p})^{r}]^{\frac{1}{r}}      \notag \\
&\leq \frac{1}{2} a,
\end{align}
and
\begin{align}\label{lsp7}
&\|e^{t\Delta}u_0\|_{L^{2}_{T_1}\dot{B}^{\frac{d}{p}}_{p,r}}                                    \notag \\
&\leq [\sum_{j\in\mathbb{Z}}(2^{\frac{d}{p}j}\|\dot{\Delta}_ju_0\|_{L^p}(\int_{0}^{T_1}e^{-2^{2^j}t}dt)^{\frac{1}{2}})^{r}]^{\frac{1}{r}}\notag \\
&\leq [\sum_{|j|\leq j_0}(2^{\frac{d}{p}j}T^{\frac{1}{2}}_1\|\dot{\Delta}_ju_0\|_{L^p})^{r}]^{\frac{1}{r}}
+[\sum_{|j|> j_0}(2^{(\frac{d}{p}-1)j}\|\dot{\Delta}_ju_0\|_{L^p}(1-e^{2^{2j}T_0})^{\frac{1}{2}})^{r}]^{\frac{1}{r}}\notag \\
&\leq 2^{j_0}T^{\frac{1}{2}}_1\|u_0\|_{\dot{B}^{\frac{d}{p}-1}_{p,r}}+ [\sum_{|j|> j_0}(2^{(\frac{d}{p}-1)j}\|\dot{\Delta}_ju_0\|_{L^p})^{r}]^{\frac{1}{r}}      \notag \\
&\leq \frac{1}{2} a,
\end{align}
where $T_0=\frac{a}{4}\frac{1}{2^{2j_0}E_0}$ and $T_1=\frac{a^2}{4^2}\frac{1}{2^{2j_0}E^2_0}$. Letting $T=\min\{T_0,T_1\}$, we get that
$\|e^{t\Delta}u_0\|_A\leq a ,$
which implies (\ref{lsp2}).

All in all, we can take
\begin{equation}\label{lsp8}
  T=\left\{\begin{array}{l}
    \infty,\quad\quad\quad \quad\quad \,\text{if } E_0\leq\frac{1}{16C_1},   \\
    \min\{T_0,T_1\},\quad \text{if } E_0\geq\frac{1}{16C_1},
  \end{array}\right.
\end{equation}
which satisfies (\ref{lsp2}).

Therefore we have
\begin{align}
&(H_1):\|u^{n+1}\|_{\widetilde{L}^{\infty}B^{\frac{d}{p}-1}_{p,r}}\leq 2E_0,\notag\\
&(H_2):\|u^{n+1}\|_{A}\leq 2a,\quad A:={\widetilde{L}^{2}_T(\dot{B}^{\frac{d}{p}}_{p,r})\cap \widetilde{L}^{1}_T(\dot{B}^{\frac{d}{p}+1}_{p,r})},\notag\\
\end{align}
where $a$ is a positive small quantity satisfying (\ref{lsp2}). This implies that the approximate sequence $(u^n,b^n)$ is uniformly bounded in $E^p_T$ in $[0,T]$.

\begin{rema}\label{j0}
By (\ref{lsp8}), we know that if $E_0\leq\frac{1}{16C_1}$, the local time $T$ depends only on $E_0$. However, if $E_0\geq\frac{1}{16C_1}$, the lifespan $T$
depends on both $E_0$ and the $j_0$ which satisfies (\ref{lsp5.5}). This means that the lifespan depends not only on the norm of the Littlewood-Paley decomposition of
the initial data $u_0$, but also on its Besov norm.

For example, set the large initial data $u_0$ such that
\begin{equation}
  \dot{\Delta}_ju_0=\left\{\begin{array}{l}
    C\phi,\quad when\quad j=0,   \\
    0,\quad\quad when\quad j\neq 0,
  \end{array}\right.
\end{equation}
and $v_0$ such that
\begin{equation}
  \dot{\Delta}_jv_0=\left\{\begin{array}{l}
    \phi,\quad when\quad j=N,   \\
    0,\quad when\quad j\neq N,
  \end{array}\right.
\end{equation}
where $C$ and $N$ are large integers such that $C=2^{N(\frac{d}{p}-1)}$ $(\forall N\in \mathbb{N})$, $\phi\in \mathbf{S}(\mathbb{R}^d)$ such that $C\|\phi\|_{L^p}>\frac{1}{16C_1}$.
Then using the argument for \eqref{lsp8} we have
$$\|u_0\|_{\dot{B}^{\frac{d}{p}-1}_{p,r}}=\|v_0\|_{\dot{B}^{\frac{d}{p}-1}_{p,r}}=C\|\phi\|_{L^p}>\frac{1}{16C_1}>\frac{a}{4}.$$

On the other hand, it's easy to see that $j_{u_0}=1$ and $j_{v_0}=N+1$ are the smallest integers such that
$$[\sum_{|j|\geq j_{u_0}}(\|\dot{\Delta}_ju_0\|_{L^p}2^{(\frac{d}{p}-1)j})^{r}]^{\frac{1}{r}}< \frac{a}{4},\quad [\sum_{|j|\geq j_{v_0}}(\|\dot{\Delta}_jv_0\|_{L^p}2^{(\frac{d}{p}-1)j})^{r}]^{\frac{1}{r}}< \frac{a}{4}.$$
Therefore, by \eqref{lsp8} we have
$$T_{u_0}=\min\{\frac{a}{4}\frac{1}{2^{2}E_0},\frac{a^2}{4^2}\frac{1}{2^{2}E^2_0}\}>T_{v_0}=\min\{\frac{a}{4}\frac{1}{2^{2(N+1)}E_0},\frac{a^2}{4^2}\frac{1}{2^{2(N+1)}E^2_0}\},$$
where $E_0=\|u_0\|_{\dot{B}^{\frac{d}{p}-1}_{p,r}}=\|v_0\|_{\dot{B}^{\frac{d}{p}-1}_{p,r}}.$
\end{rema}
\quad\\
\textbf{Step 3: Existence of a solution.}
\quad\\

Following the argument of \cite{book}, one can use the fixed point theorem to deal with the estimation of a solution. Setting $w^{n+1}=u^{n+1}-u^n$, then we have
\begin{equation}\label{lsp9}
\left\{\begin{array}{lll}
w^{n+1}_t-\Delta w^{n+1}=-\mathbb{P}div[w^{n}\otimes u^n+u^{n-1}\otimes w^{n}],\\
divw^{n+1}=0,\\
u|_{t=0}=0.
\end{array}\right.
\end{equation}
By Lemma \ref{heat}, we have
\begin{align}\label{lsp10}
\|w^{n+1}\|_{\widetilde{L}^{\frac{4}{3}}_T(\dot{B}^{\frac{d}{p}+\frac{1}{2}}_{p,r})}+\|w^{n+1}\|_{\widetilde{L}^{4}_T(\dot{B}^{\frac{d}{p}-\frac{1}{2}}_{p,r})}&\leq C(\|u^n\|_{\widetilde{L}^{4}_T(\dot{B}^{\frac{d}{p}-\frac{1}{2}}_{p,r})}\|w\|_{\widetilde{L}^{\frac{4}{3}}_T(\dot{B}^{\frac{d}{p}+\frac{1}{2}}_{p,r})}
+\|w\|_{\widetilde{L}^{4}_T(\dot{B}^{\frac{d}{p}-\frac{1}{2}}_{p,r})}\|u^{n-1}\|_{\widetilde{L}^{\frac{4}{3}}_T(\dot{B}^{\frac{d}{p}+\frac{1}{2}}_{p,r})})\notag\\
&\leq C_1(2E_0)^{\frac{1}{2}}(2a)^{\frac{1}{2}}\|w^n\|_{\widetilde{L}^{\frac{4}{3}}_T(\dot{B}^{\frac{d}{p}+\frac{1}{2}}_{p,r})}+2aC_1\|w^n\|_{\widetilde{L}^{4}_T(\dot{B}^{\frac{d}{p}-\frac{1}{2}}_{p,r})}\notag\\
&\leq \frac{1}{2}(\|w^n\|_{\widetilde{L}^{\frac{4}{3}}_T(\dot{B}^{\frac{d}{p}+\frac{1}{2}}_{p,r})}+\|w^n\|_{\widetilde{L}^{4}_T(\dot{B}^{\frac{d}{p}-\frac{1}{2}}_{p,r})}),
\end{align}
where $C_1(2E_0)^{\frac{1}{2}}(2a)^{\frac{1}{2}}\leq\frac{1}{2}$ and $2aC_1\leq \frac{1}{2}$ by \eqref{lsp2}. Thus we obtain a Cauchy sequence $\{w^n\}_{n\in\mathbb{N}}$ in $\widetilde{L}^{\frac{4}{3}}_T(\dot{B}^{\frac{d}{p}+\frac{1}{2}}_{p,r})\cap \widetilde{L}^{4}_T(\dot{B}^{\frac{d}{p}-\frac{1}{2}}_{p,r})$.

Combining the Fatou property for Besov spaces with $u^n$ being uniformly bounded in $E^p_T:= \widetilde{L}^{\infty}([0,T];\dot{B}^{\frac{d}{p}-1}_{p,r}(\mathbb{R}^d))\cap \widetilde{L}^1([0,T];\dot{B}^{\frac{d}{p}+1}_{p,r}(\mathbb{R}^d))$, we readily get
$u\in E^p_T.$

Finally, following the arguments of Theorem 2.94 and Theorem 3.19 in \cite{book}, it is a routine process to verify that $u$ satisfies the system (\ref{nssp0}). Moreover, we have $u\in C([0,T];\dot{B}^{\frac{d}{p}-1}_{p,r}(\mathbb{R}^d))$. In fact, for any $\epsilon>0$, Since $u_0\in \dot{B}^{\frac{d}{p}-1}_{p,r}$, we let $j_0\in\mathbb{N}$ such that $\sum_{|j|\geq j_0}[\|\dot{\delta}_ju_0\|2^{\frac{d}{p}-1}]^r\leq\frac{\epsilon}{2}$. Then there exists a $\delta:=\frac{\epsilon}{2\|u_0\|_{\dot{B}^{\frac{d}{p}-1}_{p,r}}^r2^{2rj_0}}$, for any $t\leq \delta$ we have
\begin{align}\label{lsp10}
\|e^{t\Delta}u_0-u_0\|^r_{\dot{B}^{\frac{d}{p}-1}_{p,r}}&\leq \sum_{j\in\mathbb{Z}}[(1-e^{-2^{2j}t})\|\dot{\delta}_ju_0\|2^{\frac{d}{p}-1}]^r \notag\\
&\leq \sum_{j\leq j_0}[(1-e^{-2^{2j}t})\|\dot{\delta}_ju_0\|2^{\frac{d}{p}-1}]^r+\sum_{|j|\geq j_0}[\|\dot{\delta}_ju_0\|2^{\frac{d}{p}-1}]^r\notag\\
&\leq 2^{2rj_0}t\|u_0\|_{\dot{B}^{\frac{d}{p}-1}_{p,r}}^r+\frac{\epsilon}{2}\notag\\
&\leq \frac{\epsilon}{2}+\frac{\epsilon}{2}=\epsilon.
\end{align}
Thus, we can easily deduce that $\lim_{t\rightarrow 0}\|u-u_0\|_{\dot{B}^{\frac{d}{p}-1}_{p,r}}=0.$
\quad\\
\textbf{Step 4: Uniqueness.}\\

To prove the uniqueness, we let $u,v\in E^p_T$ be two solutions of (\ref{nssp0}) with the same initial data. Let $w=u-v$. By the above argument of the fixed point theorem we easily get that
\begin{align}\label{lsp10}
\|w\|_{\widetilde{L}^{\frac{4}{3}}_T(\dot{B}^{\frac{d}{p}+\frac{1}{2}}_{p,r})}+\|w\|_{\widetilde{L}^{4}_T(\dot{B}^{\frac{d}{p}-\frac{1}{2}}_{p,r})}\leq \frac{1}{2}\|w\|_{\widetilde{L}^{\frac{4}{3}}_T(\dot{B}^{\frac{d}{p}+\frac{1}{2}}_{p,r})}+\frac{1}{2}\|w\|_{\widetilde{L}^{4}_T(\dot{B}^{\frac{d}{p}-\frac{1}{2}}_{p,r})}.
\end{align}
This implies that $u=v$, which completes the proof of uniqueness.

\textbf{Step 5: Continuous dependence.}\\

let $T$ be a lifespan corresponding to the initial data $u_0$ by (\ref{lsp8}). Before proving the continuous dependence of (\ref{nssp0}), we first need to prove that if $u^n_0$ tends to $u_0$ in $\dot{B}^{\frac{d}{p}-1}_{p,r}$, then there exists a lifespan $T^n$ corresponding to $u^n_0$ such that $T^n\rightarrow T$. This implies a common lifespan both for $u^n$ and $u$ when $n$ is sufficient large. Here is the key lemma:


\begin{lemm}\label{gj}
Let $u_0\in \dot{B}^{\frac{d}{p}-1}_{p,r}$ be the initial data of (\ref{nssp0}) with $1\leq r,p<\infty$. If there exist other initial data $u^n_0\in \dot{B}^{\frac{d}{p}-1}_{p,r}$ such that $\|u^n_0-u_0\|_{\dot{B}^{\frac{d}{p}-1}_{p,r}}\rightarrow 0\quad (n\rightarrow\infty)$, then we can construct a lifespan $T^n$ corresponding to $u^n_0$ such that
$$T^n\rightarrow T,\quad\quad n\rightarrow\infty ,$$
where the lifespan $T$ corresponds to $u_0$.
\end{lemm}
\begin{proof}
By virtue of Remark \ref{j0}, since $T=\infty$ when $E_0\leq\frac{1}{16C_1}$, we only consider the large initial data. Thus, we need to prove $T^n\rightarrow T$ when $E_0>\frac{1}{16C_1}$. For convenience, since $T=\min\{T_0,T_1\}$, we write down the definitions of $T_0$ and $T_1$:
$$T_0=\frac{a}{4}\frac{1}{2^{2j_0}E_0},\quad T_1=\frac{a^2}{4^2}\frac{1}{2^{2j_0}E^2_0},$$
where $j_0$ is a fixed integer such that:
$$[\sum_{|j|\geq j_0}(\|\dot{\Delta}_j u_0\|_{L^p}2^{(\frac{d}{p}-1)j})^r]^{\frac{1}{r}}< \frac{a}{4}.$$
As $u_0\in \dot{B}^{\frac{d}{p}-1}_{p,r},$ we can suppose that $j_0$ is the smallest integer such that the above inequality holds true. Since $E^n_0\rightarrow E_0$, in order to prove $T^n_0\rightarrow T_0$ and $T^n_1\rightarrow T_1$, it suffices to show that there exists a corresponding sequence $j^n_0$ satisfying
$$[\sum_{|j|\geq j^n_0}(\|\dot{\Delta}_j u^n_0\|_{L^p}2^{(\frac{d}{p}-1)j})^r]^{\frac{1}{r}}< \frac{a}{4},$$
and $j^n_0\rightarrow j_0$.

Now we begin to construct a subsequence $\{\bar{j}^m_0\}_{m\in\mathbb{N}}\geq j_0$. Fix a positive constant $\epsilon=\frac{a}{8}<\frac{a}{4}$.\\

For this $\epsilon$, there exists $N_{\epsilon}$ such that when $n\geq N_{\epsilon}$, we have
$$\|u^n_0-u_0\|_{\dot{B}^{\frac{d}{p}-1}_{p,r}}\leq \epsilon .$$
Then we let $j^{\epsilon}_0$ be the smallest integer that $$[\sum_{|j|\geq j^{\epsilon}_0}(\|\dot{\Delta}_j u_0\|_{L^p}2^{(\frac{d}{p}-1)j})^r]^{\frac{1}{r}}< \frac{a}{4}-\epsilon .$$
By the definition of $j_0$, we have $j_0\leq j^{\epsilon}_0$.


Replacing $\epsilon$ by $\frac{\epsilon}{m}$ ($m\in\mathbb{N}^+$), there exists $N_{\frac{\epsilon}{m}}$ such that when $n\geq N_{\frac{\epsilon}{m}}$, we have
$$\|u^n_0-u_0\|_{\dot{B}^{\frac{d}{p}-1}_{p,r}}\leq \frac{\epsilon}{m}.$$
Then we define that $j^{\frac{\epsilon}{m}}_0$ be the smallest integer that $$[\sum_{|j|\geq j^{\frac{\epsilon}{m}}_0}(\|\dot{\Delta}_j u_0\|_{L^2}2^{(\frac{d}{2}-1)j})^r]^{\frac{1}{r}}
< \frac{a}{4}-\frac{\epsilon}{m}.$$
Since $\frac{a}{4}-\frac{\epsilon}{m}>\frac{a}{4}-\frac{\epsilon}{m-1}$, it follows that $$j_0\leq j^{\frac{\epsilon}{m}}_0\leq j^{\frac{\epsilon}{m-1}}_0 .$$.

Now letting $\bar{j}^m_0:=j^{\frac{\epsilon}{m}}_0$, we deduce that when $n\geq N_{\frac{\epsilon}{m}}$,
\begin{align}\label{jianqie}
[\sum_{|j|\geq\bar{j}^m_0}(\|\dot{\Delta}_j u^n_0\|_{L^p}2^{(\frac{d}{p}-1)j})^r]^{\frac{1}{r}}\leq \|u^n_0-u_0\|_{\dot{B}^{\frac{d}{p}-1}_{p,r}}+[\sum_{|j|>\bar{j}^m_0}(\|\dot{\Delta}_j u_0\|_{L^p}2^{(\frac{d}{p}-1)j})^r]^{\frac{1}{r}}
<\frac{\epsilon}{m}+\frac{a}{4}-\frac{\epsilon}{m}=\frac{a}{4}.
\end{align}
Since $\{\bar{j}^m_0\}$ is a monotone and bounded sequence, we deduce that $\bar{j}^m_0\rightarrow \bar{j}_0$ $(m\rightarrow\infty)$ for some integer $\bar{j}_0\geq j_0$. For $0<\bar{\epsilon}<1$ there exists $N$ such that when $m\geq N$ we have
$$|\bar{j}^m_0-\bar{j}_0|\leq\bar{\epsilon}<1.$$
Noting that $\bar{j}^m_0,\bar{j}_0\in\mathbb{N}$, we deduce that $\bar{j}^{m}_0=\bar{j}_0$ when $m\geq N$ and $\bar{j}_0$ is the smallest integer that
$$[\sum_{|j|\geq\bar{j}_0}(\|\dot{\Delta}_ju_0\|_{L^p}2^{(\frac{d}{p}-1)j})^r]^{\frac{1}{r}}< \frac{a}{4}-\frac{\epsilon}{m}.$$
We claim that $\bar{j}_0=j_0$. Otherwise, if $\bar{j}_0>j_0$, we deduce from the above inequality that
$$[\sum_{|j|\geq j_0}(\|\dot{\Delta}_ju_0\|_{L^p}2^{(\frac{d}{p}-1)j})^r]^{\frac{1}{r}}\geq \frac{a}{4}-\frac{\epsilon}{m},\quad \forall m\geq N.$$
Since the left hand-side of the above inequality is independent of $m$, we have
$$[\sum_{|j|\geq j_0}(\|\dot{\Delta}_ju_0\|_{L^p}2^{(\frac{d}{p}-1)j})^r]^{\frac{1}{r}}\geq \frac{a}{4},\quad\forall m\geq N.$$
This contradicts the definition of $j_0$. So we have $\bar{j}^m_0\rightarrow \bar{j}_0=j_0$ $(m\rightarrow\infty)$.

Finally, we construct a sequence $j^n_0$ by $\bar{j}^m_0$ when $n\geq N_{\epsilon}$:
\begin{equation}\label{equ1-1-5}
  j^n_0:=\left\{\begin{array}{l}
    \bar{j}^1_0, \quad N_{\epsilon}\leq n<N_{\frac{\epsilon}{2}},\\
    \bar{j}^2_0, \quad N_{\frac{\epsilon}{2}}\leq n< N_{\frac{\epsilon}{3}},\\
    ... ...\\
    \bar{j}^m_0, \quad N_{\frac{\epsilon}{m}}\leq n< N_{\frac{\epsilon}{m+1}},\\
    ... ...\\
  \end{array}\right.
\end{equation}
By virtue of \eqref{jianqie}, one can check that
$$[\sum_{|j|\geq j^n_0}(\|\dot{\Delta}_ju^n_0\|_{L^p}2^{(\frac{d}{p}-1)j})^r]^{\frac{1}{r}}< \frac{a}{4}.$$
Using the monotone bounded theorem, on can prove that $j^n_0\rightarrow j_0$ $(n\rightarrow\infty)$. Therefore, we have
$$T^n_0\rightarrow T_0,\quad T^n_1\rightarrow T_1\quad\Longrightarrow T^n\rightarrow T,\quad n\rightarrow\infty .$$
This completes the proof of the lemma.
\end{proof}
\begin{rema}
By Lemma \ref{gj}, letting $T^{\infty}$ be the lifespan of $u^{\infty}$, then we can find a $T^n$ corresponding to $u^n$ such that $T^n\rightarrow T^{\infty}$ when $n\rightarrow\infty$. That is to say, for fixed some small $\delta >0$, there exists an integer $N_{\delta}$ such that when $n\geq N$, we have
$$|T^n-T^{\infty}|< \delta .$$

Thus, we can choose $T=T^{\infty}-\delta$, which is also the common lifespan both for $u^{\infty}$ and $u^{n}$, but independent of $n$.
\end{rema}
Now we begin to prove the continuous dependence.
\begin{theo}
Let $p<\infty$, $n\in\mathbb{N}\cap\{\infty\}$. Assume that $u^n$ is the solution to the system (\ref{nssp0}) with the initial data $u^n_0$. If $u^n_0$ tends to $u^{\infty}_0$ in $\dot{B}^{\frac{d}{p}-1}_{p,r}$, then there exists a positive ${T}$ independent of $n$ such that $u^n$ tends to $u^{\infty}$ in $\widetilde{L}_{T}(\dot{B}^{\frac{d}{p}-1}_{p,r})\cap \widetilde{L}^{1}_{T}(\dot{B}^{\frac{d}{p}+1}_{p,r})$.
\end{theo}

\begin{proof}
By Lemma \ref{gj}, we can easily find that $T=T^{\infty}-\delta$ is the common lifespan for both $u^{n}$ and $u^{\infty}$ for $n$ large enough. Since $u^n_0$ tends to $u^{\infty}_0$ in $\dot{B}^{\frac{d}{p}-1}_{p,r}$, by the argument of Step 2, we can easily get that
$$\|u^n\|_{\widetilde{L}^{\infty}_{T}(\dot{B}^{\frac{d}{p}-1}_{p,r})}\leq 2E_0,\quad\|u^n\|_{\widetilde{L}^{p}_{T}(\dot{B}^{\frac{d}{p}}_{p,r})\cap \widetilde{L}^{1}_{T}(\dot{B}^{\frac{d}{p}+1}_{p,r})}\leq 2a,\quad \forall n,$$
where $a$ is a positive small quantity satisfying (\ref{lsp2}).
Similar to the proof of the uniqueness, we let
\begin{equation}\label{lsp9}
\left\{\begin{array}{lll}
(u^n-u^{\infty})_t-\Delta (u^n-u^{\infty})=-\mathbb{P}div[(u^n-u^{\infty})\otimes u^n+u^{\infty}\otimes (u^n-u^{\infty})],\\
div(u^n-u^{\infty})=0,\\
(u^n-u^{\infty})|_{t=0}=(u^n_0-u^{\infty}_0).
\end{array}\right.
\end{equation}
By Lemma \ref{heat}, we have
\begin{align}\label{lsp10}
&\|u^n-u^{\infty}\|_{\widetilde{L}^{\infty}_T(\dot{B}^{\frac{d}{p}-1}_{p,r})\cap \widetilde{L}^{4}_T(\dot{B}^{\frac{d}{p}-\frac{1}{2}}_{p,r})\cap \widetilde{L}^{2}_T(\dot{B}^{\frac{d}{p}}_{p,r})\cap \widetilde{L}^{\frac{4}{3}}_T(\dot{B}^{\frac{d}{p}+\frac{1}{2}}_{p,r})\cap \widetilde{L}^{1}_T(\dot{B}^{\frac{d}{p}+1}_{p,r})}\notag\\
&\leq \|u^n_0-u^{\infty}_0\|_{\dot{B}^{\frac{d}{p}-1}_{p,r}}+
C_1(\|u\|_{\widetilde{L}^{4}_T(\dot{B}^{\frac{d}{p}-\frac{1}{2}}_{p,r})}\|u^n-u^{\infty}\|_{\widetilde{L}^{\frac{4}{3}}_T(\dot{B}^{\frac{d}{p}+\frac{1}{2}}_{p,r})}
+\|u^n-u^{\infty}\|_{\widetilde{L}^{4}_T(\dot{B}^{\frac{d}{p}-\frac{1}{2}}_{p,r})}\|v\|_{\widetilde{L}^{\frac{4}{3}}_T(\dot{B}^{\frac{d}{p}+\frac{1}{2}}_{p,r})})\notag\\
&\leq \|u^n_0-u^{\infty}_0\|_{\dot{B}^{\frac{d}{p}-1}_{p,r}}+C_1(2E_0)^{\frac{1}{2}}(2a)^{\frac{1}{2}}\|u^n-u^{\infty}\|_{\widetilde{L}^{\frac{4}{3}}_T(\dot{B}^{\frac{d}{p}+\frac{1}{2}}_{p,r})}
+2aC_1\|u^n-u^{\infty}\|_{\widetilde{L}^{4}_T(\dot{B}^{\frac{d}{p}-\frac{1}{2}}_{p,r})}\notag\\
&\leq \|u^n_0-u^{\infty}_0\|_{\dot{B}^{\frac{d}{p}-1}_{p,r}}+\frac{1}{2}(\|u^n-u^{\infty}\|_{\widetilde{L}^{\frac{4}{3}}_T(\dot{B}^{\frac{d}{p}+\frac{1}{2}}_{p,r})}
+\frac{1}{2}\|u^n-u^{\infty}\|_{\widetilde{L}^{4}_T(\dot{B}^{\frac{d}{p}-\frac{1}{2}}_{p,r})}),
\end{align}
where $C_1(2E_0)^{\frac{1}{2}}(2a)^{\frac{1}{2}}\leq\frac{1}{2}$ and $2aC_1\leq \frac{1}{2}$ are based on \eqref{lsp2}. Then we have
\begin{align}\label{lsp11}
\|u^n-u^{\infty}\|_{L^{\infty}_T\dot{B}^{\frac{d}{p}-1}_{p,r}\cap L^{1}_T\dot{B}^{\frac{d}{p}+1}_{p,r}}
\leq C\|u^n_0-u^{\infty}_0\|_{\dot{B}^{\frac{d}{p}-1}_{p,r}}\rightarrow 0,\quad n\rightarrow \infty.
\end{align}
This completes the proof of continuous dependence.

Finally, combining Steps 1-2, we obtain the local well-posedness of the (NS) equation. This completes the proof of Theorem \ref{theorem}.

\end{proof}

Similarly, we give the proof of Theorem \ref{theorem2}.\\
\textbf{The proof of Theorem \ref{theorem2}:}
\begin{proof}
For $u_0\in\bar{B}^{\frac{d}{p}-1}_{p,\infty}$, By the definition of $\bar{B}^{\frac{d}{p}-1}_{p,\infty}$, there exists $j_0\in\mathbb{N}$ such that
$$\lim_{|j|\geq j_0}2^{\frac{d}{p}-1}\|\dot{\Delta}_jf\|_{L^p}\leq\frac{1}{4} a.$$
Then letting $T_0=\frac{a}{4}\frac{1}{2^{2j_0}E_0}$ and $T_1=\frac{a^2}{4^2}\frac{1}{2^{2j_0}E^2_0}$, we have
\begin{align}
&\|e^{t\Delta}u_0\|_{\widetilde{L}^{1}_{T_0}\dot{B}^{\frac{d}{p}+1}_{p,\infty}}                                    \notag \\
&\leq \sup_{j\in\mathbb{Z}}2^{(\frac{d}{p}+1)j}\|\dot{\Delta}_ju_0\|_{L^p}\int_{0}^{T_0}e^{-2^{2^j}t}dt\notag \\
&\leq \sup_{|j|\leq j_0}2^{(\frac{d}{p}+1)j}\|\dot{\Delta}_ju_0\|_{L^p}T_0
+\sup_{|j|> j_0}2^{(\frac{d}{p}-1)j}(1-e^{2^{2j}T_0})\|\dot{\Delta}_ju_0\|_{L^p}\notag \\
&\leq 2^{2j_0}{T_0}\|u_0\|_{\dot{B}^{\frac{d}{p}-1}_{p,r}}+ \sup_{|j|> j_0}2^{(\frac{d}{p}-1)j}\|\dot{\Delta}_ju_0\|_{L^p}      \notag \\
&\leq \frac{1}{2} a,
\end{align}
and
\begin{align}\label{lsp7}
&\|e^{t\Delta}u_0\|_{L^{2}_{T_1}\dot{B}^{\frac{d}{p}}_{p,\infty}}                                    \notag \\
&\leq \sup_{j\in\mathbb{Z}}(2^{\frac{d}{p}j}\|\dot{\Delta}_ju_0\|_{L^p}(\int_{0}^{T_1}e^{-2^{2^j}t}dt)^{\frac{1}{2}}\notag \\
&\leq \sup_{|j|\leq j_0}2^{\frac{d}{p}j}T^{\frac{1}{2}}_1\|\dot{\Delta}_ju_0\|_{L^p}
+\sup_{|j|> j_0}2^{(\frac{d}{p}-1)j}\|\dot{\Delta}_ju_0\|_{L^p}(1-e^{2^{2j}T_0})^{\frac{1}{2}}\notag \\
&\leq 2^{j_0}T^{\frac{1}{2}}_1\|u_0\|_{\dot{B}^{\frac{d}{p}-1}_{p,r}}+ \sup_{|j|> j_0}2^{(\frac{d}{p}-1)j}\|\dot{\Delta}_ju_0\|_{L^p}      \notag \\
&\leq \frac{1}{2} a.
\end{align}
Setting $T=\min\{T_0,T_1\}$, the proof is then similar to that of Theorem \ref{theorem}. So we can prove that the Cauchy problem (\ref{nssp0}) is locally well-posed in the Hadamard sense.
\end{proof}

Finally, we consider the Leray weak solutions with extra conditions on the initial data.\\
\textbf{The proof of Theorem \ref{theorem3}:}
\begin{proof}
Set $u_0:=u^{\infty}_0$. Applying Theorems 1-2, we can obtain some local well-posed solutions $u^n$ ($n\in\mathbb{N}\cap\{\infty\}$) in $E^p_T$ (or $\dot{E}_T$) for a common lifespan $T=T_{u_0}$ independent of $n$. They are also the Leray weak solutions. Moreover, for any $n\in\mathbb{N}\cap\{\infty\}$ we have
\begin{align}\label{L2lsp2}
&\quad \|u^{n}\|_{\widetilde{L}^{\infty}_T(B^{\frac{d}{p}-1}_{p,1})}\leq 2E_0,\notag\\
&\quad \|u^{n}\|_{A}\leq 2a,\quad A:={\widetilde{L}^{2}_T(\dot{B}^{\frac{d}{p}}_{p,r})\cap \widetilde{L}^{1}_T(\dot{B}^{\frac{d}{p}+1}_{p,r})},
\end{align}
and
\begin{align}\label{L2lsp3}
a\leq \min\{\frac{1}{16C^2_1E_0},(\frac{\sqrt{E_0}}{4C_1})^{\frac{2}{3}},\frac{1}{4C_1}\}\leq \frac{1}{4C_1},
\end{align}
where $E_0:=\|u_0\|_{\dot{B}^{\frac{d}{p}-1}_{p,r}}$, $C_1\geq C$, $C$ is the constant of Lemma \ref{heat}. In fact, we can choose $C_1$ much bigger such as $C_1=48C$.

Setting $w=u^n-u^{\infty}$, then we have
\begin{equation}\label{L2lsp9}
\left\{\begin{array}{lll}
w_t-\Delta w=-\mathbb{P}div[w\otimes u^n+u^{\infty}\otimes w],\\
divw=0,\\
w|_{t=0}=u^n_0-u_0.
\end{array}\right.
\end{equation}
By Lemma \ref{heat}, we have
\begin{align}\label{L2lsp10}
\|w\|_{\widetilde{L}^{\infty}_T(\dot{B}^{0}_{2,2})}+\|w\|_{\widetilde{L}^{2}_T(\dot{B}^{1}_{2,2})}
\leq C(\|u^n_0-u_0\|_{L^2}+\|w\otimes u^n\|_{\widetilde{L}^{2}_T(\dot{B}^{0}_{2,2})}+\|u^{\infty}\otimes w\|_{\widetilde{L}^{2}_T(\dot{B}^{0}_{2,2})}).
\end{align}
Applying the Bony decomposition, we have:\\
(1) Setting $\frac{1}{2}=\frac{1}{p_1}+\frac{1}{p'_1}$ with $2<p_1,p'_1<\infty$, then
\begin{align}\label{L2lsp11}
\|T_{u^n}w\|_{\widetilde{L}^{2}_T(\dot{B}^{0}_{2,2})}&\leq \|u^n\|_{\widetilde{L}^{p'_1}_T(\dot{B}^{\frac{2}{p'_1}-1}_{\infty,\infty})}\|w\|_{\widetilde{L}^{p_1}_T(\dot{B}^{1-\frac{2}{p'_1}}_{2,2})}\notag\\
&\leq \|u^n\|_{\widetilde{L}^{p'_1}_T(\dot{B}^{\frac{2}{p'_1}-1+\frac{d}{p}}_{p,\infty})}\|w\|_{\widetilde{L}^{p_1}_T(\dot{B}^{\frac{2}{p_1}}_{2,2})}\notag\\
&\leq \|u^n\|_{\widetilde{L}^{p'_1}_T(\dot{B}^{\frac{2}{p'_1}-1+\frac{d}{p}}_{p,\infty})}(\|w\|_{\widetilde{L}^{\infty}_T(\dot{B}^{0}_{2,2})}+\|w\|_{\widetilde{L}^{2}_T(\dot{B}^{1}_{2,2})}).
\end{align}
\quad\\
(2) Setting $\frac{1}{2}=\frac{1}{p_2}+\frac{1}{p'_2}$ with $2<p_2,p'_2<\infty$ and $p\leq p_2< \frac{d}{2}p_1<\infty$, then
\begin{align}\label{L2lsp12}
\|T_{w}u^n\|_{\widetilde{L}^{2}_T(\dot{B}^{0}_{2,2})}&\leq\|u^n\|_{\widetilde{L}^{p'_1}_T(\dot{B}^{\frac{d}{p_2}-\frac{2}{p_1}}_{p_2,\infty})}\|w\|_{\widetilde{L}^{p_1}_T(\dot{B}^{\frac{2}{p_1}-\frac{d}{p_2}}_{p'_2,2})}\notag\\
&= \|u^n\|_{\widetilde{L}^{p'_1}_T(\dot{B}^{\frac{d}{p_2}+\frac{2}{p'_1}-1}_{p_2,\infty})}\|w\|_{\widetilde{L}^{p_1}_T(\dot{B}^{\frac{2}{p_1}}_{2,2})}\notag\\
&\leq \|u^n\|_{\widetilde{L}^{p'_1}_T(\dot{B}^{\frac{d}{p}+\frac{2}{p'_1}-1}_{p,\infty})}(\|w\|_{\widetilde{L}^{\infty}_T(\dot{B}^{0}_{2,2})}+\|w\|_{\widetilde{L}^{2}_T(\dot{B}^{1}_{2,2})}).
\end{align}
\quad\\
(3) Setting $\frac{1}{p_4}=\frac{1}{2}+\frac{1}{p_3}$ with $p\leq p_3<\infty$ and $p_4<2$, then
\begin{align}\label{L2lsp13}
\|R(w,u^n)\|_{\widetilde{L}^{2}_T(\dot{B}^{0}_{2,2})}&\leq \|R(w,u^n)\|_{\widetilde{L}^{2}_T(\dot{B}^{\frac{d}{p_3}}_{p_4,2})}\notag\\
&\leq \|u^n\|_{\widetilde{L}^{p'_1}_T(\dot{B}^{\frac{d}{p_3}-\frac{2}{p_1}}_{p_3,\infty})}\|w\|_{\widetilde{L}^{p_1}_T(\dot{B}^{\frac{2}{p_1}}_{2,2})}\notag\\
&\leq \|u^n\|_{\widetilde{L}^{p'_1}_T(\dot{B}^{\frac{d}{p}-\frac{2}{p_1}}_{p,\infty})}\|w\|_{\widetilde{L}^{p_1}_T(\dot{B}^{\frac{2}{p_1}}_{2,2})}\notag\\
&\leq \|u^n\|_{\widetilde{L}^{p'_1}_T(\dot{B}^{\frac{d}{p}+\frac{2}{p'_1}-1}_{p,\infty})}(\|w\|_{\widetilde{L}^{\infty}_T(\dot{B}^{0}_{2,2})}+\|w\|_{\widetilde{L}^{2}_T(\dot{B}^{1}_{2,2})}).
\end{align}
Combining interpolation inequality with \eqref{L2lsp2}, we deduce that
$$\|u^n\|_{\widetilde{L}^{p'_1}_T(\dot{B}^{\frac{d}{p}+\frac{2}{p'_1}-1}_{p,\infty})}\leq \|u^n\|^{\frac{2}{p'_1}}_{\widetilde{L}^{2}_T(\dot{B}^{\frac{d}{p}}_{p,\infty})}\|u^n\|^{1-\frac{2}{p'_1}}_{\widetilde{L}^{\infty}_T(\dot{B}^{\frac{d}{p}+1}_{p,\infty})}\leq (2a)^{\frac{2}{p'_1}}2E_0^{1-\frac{2}{p'_1}}.$$
Set $f(z)=(2a)^{\frac{2}{z}}2E_0^{1-\frac{2}{z}}$. Obviously, it is a continuous function. It's easy to deduce that $\lim_{z\rightarrow 2}f(z)=2a$. So for any given $1\leq p<\infty$, we let $p\leq p_2< \frac{d}{2}p_1<\infty$ and $p_1$ sufficiently large such that $p'_1$ is sufficiently close to $2$. Then we have $f(p'_1)=(2a)^{\frac{2}{p'_1}}2E_0^{1-\frac{2}{p'_1}}\leq 2\cdot 2a$ for this $p'_1$.
Combining \eqref{L2lsp11}, \eqref{L2lsp12} and \eqref{L2lsp13} yields that
\begin{align}\label{L2lsp14}
\|w\otimes u^n\|_{\widetilde{L}^{2}_T(\dot{B}^{0}_{2,2})}&\leq 3\|u^n\|_{\widetilde{L}^{p'_1}_T(\dot{B}^{\frac{d}{p}+\frac{2}{p'_1}-1}_{p,\infty})}(\|w\|_{\widetilde{L}^{\infty}_T(\dot{B}^{0}_{2,2})}+\|w\|_{\widetilde{L}^{2}_T(\dot{B}^{1}_{2,2})})\notag\\
&\leq 12a(\|w\|_{\widetilde{L}^{\infty}_T(\dot{B}^{0}_{2,2})}+\|w\|_{\widetilde{L}^{2}_T(\dot{B}^{1}_{2,2})}).
\end{align}
Similarly for $\|u^{\infty}\otimes w\|_{\widetilde{L}^{2}_T(\dot{B}^{0}_{2,2})}$, thus we have
\begin{align}
\|w\|_{\widetilde{L}^{\infty}_T(\dot{B}^{0}_{2,2})}+\|w\|_{\widetilde{L}^{2}_T(\dot{B}^{1}_{2,2})}
&\leq C(\|u^n_0-u_0\|_{L^2}+\|w\otimes u^n\|_{\widetilde{L}^{2}_T(\dot{B}^{0}_{2,2})}+\|u^{\infty}\otimes w\|_{\widetilde{L}^{2}_T(\dot{B}^{0}_{2,2})})\notag\\
&\leq C\|u^n_0-u_0\|_{L^2}+24aC(\|w\|_{\widetilde{L}^{\infty}_T(\dot{B}^{0}_{2,2})}+\|w\|_{\widetilde{L}^{2}_T(\dot{B}^{1}_{2,2})})\notag\\
&\leq C\|u^n_0-u_0\|_{L^2}+\frac{1}{8}(\|w\|_{\widetilde{L}^{\infty}_T(\dot{B}^{0}_{2,2})}+\|w\|_{\widetilde{L}^{2}_T(\dot{B}^{1}_{2,2})}).
\end{align}
Since $\widetilde{L}^{\infty}_T(\dot{B}^{0}_{2,2})\hookrightarrow {L}^{\infty}_T(L^2)$, we finally obtain \eqref{L2}.
\end{proof}
\quad\\

\noindent\textbf{Acknowledgements.} This work was partially supported by NNSFC [grant numbers 11671407 and 11701586],  FDCT (grant number 0091/2018/A3), Guangdong Special Support Program (grant number 8-2015), and the key project of NSF of  Guangdong province (grant number 2016A030311004).
\phantomsection
\addcontentsline{toc}{section}{\refname}


\end{document}